\documentclass{article}
\usepackage{amsmath}
\usepackage{amsfonts}
\usepackage{amsthm}
\newtheorem{definition}{Definition}
\newtheorem{proposition}{Proposition}
\newtheorem{theorem}{Theorem}
\newtheorem{lemma}{Lemma}

\newtheorem{corollary}{Corollary}
\theoremstyle{remark}
\newtheorem{remark}{Remark}

\newtheorem{conjecture}{Conjecture}
\title{On structure of homogenenous Wick ideals in Wick $*$-algebras with braided coefficients}
\author{Vasyl Ostrovskyi$^\text{(a)}$
\and Danil Proskurin$^\text{(b)}$
\and Yurii Savchuk$^\text{(c)}$
\and Lyudmila Turowska$^\text{(d)}$}
\date{
\footnotetext{$^\text{(a)}$ Institute of Mathematics, NAS of Ukraine, Ukraine, \texttt{vo@imath.kiev.ua}}
\footnotetext{$^\text{(b)}$ Kyiv Taras Shevchenko University, Ukraine, \texttt{prosk@univ.kiev.ua}}
\footnotetext{$^\text{(c)}$ Leipzig University, Germany, \texttt{Yuriy.Savchuk@math.uni-leipzig.de}}
\footnotetext{$^\text{(d)}$ Chalmers University of Technology and University of Gothenburg, Sweden, \texttt{turowska@chalmers.se}}
}
\begin{document}
\maketitle
\section{Introduction}

In this paper we present some results on structure of Wick homogenenous ideals of quadratic algebras allowing Wick ordering, shortly Wick algerbas, introduced in  \cite{jsw}. Namely, let $\{T_{ij}^{kl},\ i,j,k,l=1,\ldots,d\}\subset\mathbb{C}$ satisfy conditions $T_{ji}^{lk}=\overline{T}_{ij}^{kl}$, then Wick algebra $W(T)$ is generated by $a_i$, $a_i^*$, $i=1,\ldots,d$,
satisfying commutation relations of the form
\begin{equation}\label{wick}
a_i^*a_j=\delta_{ij}1+\sum_{k,l=1}^d T_{ij}^{kl}a_la_k^*,\ i,j=1,\ldots,d.
\end{equation}

Following \cite{jsw} consider finite-dimensional Hilbert space
$\mathcal{H}=\mathbb{C}\langle e_1,\cdots,e_d\rangle$ and its formal dual
$\mathcal{H}^*=\mathbb{C}\langle e_1^*,\cdots,e_d^*\rangle$, where $\{e_i,\ i=1,\ldots,d\}$ form an orthonormal base of $\mathcal{H}$. Put
$\mathcal{T}(\mathcal{H},\mathcal{H}^*)$ to be the full tensor algebra over $\mathcal{H}$ and $\mathcal{H^*}$, then
\begin{equation}\label{canonic}
W(T)\simeq \mathcal{T}(\mathcal{H},\mathcal{H}^*)/\langle e_i^*\otimes e_j-\sum_{k,l=1}^d T_{ij}^{kl} e_l\otimes e_k^*\rangle.
\end{equation}
Note, that in this realisation the free algebra generated by $a_i$, $i=1,\ldots,d$ coincides with $\mathcal{T}(\mathcal{H})=
\mathbb{C}\Omega\oplus\bigoplus_{n\in\mathbb{N}}\mathcal{H}^{\otimes n}$.

The Fock representation of $W(T)$ is defined on $\mathcal{T}(\mathcal{H})$ by the rules
\[
a_i^*\Omega=0,\quad a_i e_{i_1}\otimes\cdots \otimes e_{i_k}=e_i\otimes e_{i_1}\otimes\cdots \otimes e_{i_k},\ i=1,\ldots,d,
\]
the action of $a_i^*$, $i=1,\ldots, d$, on vectors other than $\Omega$, is determined inductively using the commutation relation in  $W(T)$.
It was proved in \cite{jsw} that there exists a unique sesquilinear form $\langle\cdot,\cdot\rangle_F$,
called the {\it Fock scalar product},
on $\mathcal{T}(\mathcal{H})$,  such that the Fock representation becomes a $*$-representation with respect to this form. It is defined in such a way that the subspaces $\mathcal{H}^{\otimes n}$ and  $\mathcal{H}^{\otimes m}$ are orthogonal if $m\ne n$ and
\[
\langle X,Y\rangle_F=\langle X, P_n Y\rangle,\quad X,Y\in\mathcal{H}^{\otimes n}.
\]
where by $\langle\cdot,\cdot\rangle$ we denote the standard scalar product on $\mathcal{H}^{\otimes n}$ and
$P_n\colon\mathcal{H}^{\otimes n}\rightarrow
\mathcal{H}^{\otimes n}$ is an operator defined in  the following way  (see \cite{jsw}):
First we introduce an operator
$T\colon\mathcal{H}^{\otimes 2}\rightarrow\mathcal{H}^{\otimes 2}$ given by
\begin{equation}\label{opert}
T e_{k}\otimes e_{l} =
\sum_{i,j=1}^d T_{ik}^{lj}e_{i}\otimes e_{j}.
\end{equation}
Note that $T$ is self-adjoint with respect to the standard scalar product on $\mathcal{H}^{\otimes 2}$. Further, for any $n>2$ consider the following extensions of $T$ to $\mathcal{H}^{\otimes n}$:
\[
T_i=\bigotimes_{k=1}^{i-1}\mathbf{1}_{\mathcal{H}}\otimes T\otimes
\bigotimes_{k=i+2}^n\mathbf{1}_{\mathcal{H}},\quad i=1,\ldots,n-1.
\]
Then we set $P_0=1$, $P_1=\mathbf{1}_{\mathcal{H}}$, $P_2=\mathbf{1}_{\mathcal{H}^{\otimes 2}}+T$ and
\begin{equation}\label{pn}
P_{n}=(\mathbf{1}_{\mathcal{H}}\otimes P_{n-1})R_n,\ n\ge 3,
\end{equation}
where
\[
R_n\colon\mathcal{H}^{\otimes n}\rightarrow\mathcal{H}^{\otimes n},\quad R_n=\mathbf{1}_{\mathcal{H}^{\otimes n}}+T_1+T_1T_2+\cdots+T_1T_2\cdots T_{n-1}.
\]
\begin{remark}
The operators $R_n$, $n\ge 2$, are used to obtain explicit formulas
for commutation relations between generators $a_i^*$, $i=1,\ldots,d$
and  homogeneous polynomial in noncommutative variables $a_1,\ldots,
a_d$. Namely, by \cite{proams}, for $X\in\mathcal{H}^n$ one has the
following  equality in $W(T)$ (here we use the canonical
realisation)
\[
e_i^*\otimes X=\mu_0(e_i^*)(R_n X+\sum_{k=1}^d T_1T_2\cdots T_n (X\otimes e_k)\otimes e_k^*),
\]
where $\mu_0(e_i^*):\mathcal T(\mathcal H)\to \mathcal T(\mathcal H)$ is given by
\[
\mu_0(e_i^*)e_{i_1}\otimes e_{i_2}\otimes\cdots e_{i_s}=\delta_{ii_1}e_{i_2}\otimes\cdots e_{i_s},\ s\ge 1,\quad \mu_0(e_i^*)\Omega =0.
\]
This allows to determine explicitly  the action of $a_i^*$ in the Fock representation as follows
\[
a_i^* X=\mu_0(e_i^*)R_n X,\quad X\in\mathcal{H}^{\otimes n}.
\]
\end{remark}

Positivity of the Fock scalar product means that  $P_n\geq 0$ for all $n\ge 2$. In this case the Fock representation can be extended to a  $*$-representation
of $W(T)$ on a Hilbert space, which is a completion of $\mathcal{T}(\mathcal{H})/\bigoplus_{n\ge 2}\ker P_n$ with respect to the norm defined by the Fock scalar product. Sufficient conditions for positivity of family $\{P_n,\ n\ge 2\}$ can be found in \cite{bs,jps, jsw}.
For instance if  $T$ is {\it braided}, i.e. $T_1T_2T_1=T_2T_1T_2$ on  $\mathcal{H}^{\otimes 3}$,  and $||T||\le 1$, then by \cite{bs} $P_n\ge 0$, $n\ge 2$. Moreover in this case for any $n\ge 2$
\[
\ker P_n=\sum_{i=1}^{n-1}\ker (\mathbf{1}_{\mathcal{H}^{\otimes n}}+T_i)
\]
and the kernel of the Fock representation is generated as a two-sided $*$-ideal by  $\ker (\mathbf{1}_{\mathcal{H}^{\otimes 2}}+T)$, see \cite{jps}.
Furthermore, if
 $T$ is braided and $\ker (\mathbf{1}_{\mathcal{H}^{\otimes 2}}+T)\ne \{0\}$, the two-sided ideal $\mathcal{I}_2\subset\mathcal{T}(\mathcal{H})$ generated by $\ker (\mathbf{1}_{\mathcal{H}^{\otimes 2}}+T)$ is invariant with respect to multiplication by any $a_i^*$, $i=1,\ldots,d$. i.e.
\begin{equation}\label{ideal}
e_i^*\otimes\mathcal{I}_2\subset\mathcal{I}_2+\mathcal{I}_2\otimes\mathcal{H}^{*}
\end{equation}
Ideals  $I\subset\mathcal{T}(\mathcal{H})$ satisfying (\ref{ideal}) are called {\it Wick ideals}, see \cite{jsw}. It was shown that homogeneous Wick ideals, i.e. those ones which are generated by subspaces in $\mathcal{H}^{\otimes n}$, are annihilated by the Fock representation, see \cite{jsw}. In \cite{jps} the authors prove that if the operator $T$ is braided then
 existence of homogeneous Wick ideals is necessary for existence of Wick ideals in general. If $T$ is a braided contraction, then any homogeneous Wick ideal of higher degree is contained in a  largest quadratic one, see \cite{jps}. Note that for some Wick algebras (e.g. Wick algebras associated with  twisted canonical commutation relations of W. Pusz and S.L. Woronowicz, see \cite{jsw,pw}; quonic commutation relations, see \cite{mar} and others) their quadratic Wick ideals are contained in their $*$-radicals, i.e. such ideals are annihilated by any bounded $*$-representation of the corresponding algebra.

In this paper we investigate the structure of homogeneous Wick ideals of higher degrees. We present a  method how to construct a homogeneous Wick ideal
$\mathcal{I}_{n+1}$ of degree $n+1$ out of a  homogeneous Wick ideal $\mathcal{I}_n$ of degree $n$ so that $\mathcal{I}_{n+1}\subset\mathcal{I}_n$. We show that in some particular cases our procedure allows to get a description of largest homogeneous Wick ideals of higher degrees having  generators of the largest quadratic Wick ideal only.
Finally we study classes of $*$-representations of Wick version of CCR annihilating certain homogeneous Wick ideals of degree higher than $2$.
\section{Wick ideals: basic definitions and properties.}
The notion of Wick ideal in quadratic Wick algebra was presented in \cite{jsw}. It was proposed as a natural way to introduce additional relations between generators $a_i$, $i=1,\ldots,d$, which are consistent with the basic relations of the algebra.

Following \cite{jsw} we will work with the canonical realisation of $W(T)$ as a  quoitient of the tensor algebra $\mathcal T(\mathcal H,\mathcal H^*)$
given by  (\ref{canonic}). In this realisation the subalgebra generated by $a_i$, $i=1,\ldots,d$, is identified with $\mathcal{T}(\mathcal{H})$.
\begin{definition}\label{defwickid}
A two-sided ideal $\mathcal{I}\subset\mathcal{T}(\mathcal{H})$
is called a Wick ideal if
\[
\mathcal{T}(\mathcal{H}^*)\otimes\mathcal{I}\subset\mathcal{I}\otimes
\mathcal{T}(\mathcal{H}^*).
\]
If the Wick ideal $\mathcal{I}$ is generated by a subspace  $\mathcal{I}_0\subset\mathcal{H}^{\otimes n}$, then  $\mathcal{I}$ is called a homogeneous Wick ideal of degree $n$.
\end{definition}
It is easy to verify the following criteria for a two-sided ideal $\mathcal{I}$ to be a Wick one, see \cite{jsw}.
\begin{proposition}
A two-sided  ideal $\mathcal{I}\subset\mathcal{T}(\mathcal{H)}$ is Wick iff
\[
\mathcal{H}^*\otimes\mathcal{I}\subset\mathcal{I}+\mathcal{I}\otimes\mathcal{H}^*.
\]
\end{proposition}
\begin{remark}
If an  ideal $\mathcal{I}\subset\mathcal{T}(\mathcal{H})$ is generated by a  subspace  $\mathcal{I}_0\subset\mathcal{H}^{\otimes n}$, then it is Wick iff
\[
\mathcal{H}^*\otimes\mathcal{I}_0\subset\mathcal{I}_0+
\mathcal{I}_0\otimes\mathcal{H}^*
\]
\end{remark}
It is  important from the representation theory point of view  to get a
 precise description of generators of homogeneous Wick ideals of degrees higher than $2$. The first step in this direction was done in \cite{proams}. Namely, in this paper the following statement was proved.
\begin{proposition}\label{cubide}
Let $T$ be a braided contraction  and let $\mathcal{I}_2\subset\mathcal{H}^{\otimes 2}$ generate the largest quadratic Wick ideal.  Then
\[
\mathcal{I}_3=(\mathbf{1}_{\mathcal{H}^{\otimes 3}}-T_1T_2)(\mathcal{I}_2\otimes\mathcal{H})
\]
generates the largest  Wick ideal of degree 3.
\end{proposition}

Below we will often say "homogeneous Wick ideal of degree $n$" meaning a
 linear subspace in $\mathcal{H}^{\otimes n}$ generating this ideal.
\section{Homogeneous Wick ideals}
We start with a simple observation, showing that the product of homogeneous Wick ideals is again a  homogeneous Wick ideal.
\begin{proposition}
Let $\mathcal{J}_n$ and $\mathcal{J}_k$ be homogeneous Wick ideals of degree $n$ and $k$ respectively, then their tensor product  $\mathcal{J}_n\otimes\mathcal{J}_k$ is a homogeneous Wick ideal of degree $n+k$.
\end{proposition}
\begin{proof}
Indeed, since for a Wick ideal one has
\[
\mathcal{H}^*\otimes\mathcal{I}\subset\mathcal{I}+\mathcal{I}\otimes\mathcal{H}^*
\]
we get
\begin{align*}
\mathcal{H}^* &\otimes(\mathcal{J}_n\otimes\mathcal{J}_k)
\subset(\mathcal{J}_n+\mathcal{J}_n\otimes\mathcal{H}^*)
\otimes\mathcal{J}_k=\mathcal{J}_n\otimes\mathcal{J}_k+
\mathcal{J}_n\otimes\mathcal{H}^*\otimes\mathcal{J}_k\subset\\
&\subset\mathcal{J}_n\otimes\mathcal{J}_k+
\mathcal{J}_n\otimes\mathcal{J}_k
+\mathcal{J}_n\otimes\mathcal{J}_k\otimes\mathcal{H}^*
=\mathcal{J}_n\otimes\mathcal{J}_k
+\mathcal{J}_n\otimes\mathcal{J}_k\otimes\mathcal{H}^*.
\end{align*}
Thus, $\mathcal{J}_n\otimes\mathcal{J}_k\subset\mathcal{H}^{\otimes (n+k) }$ is a Wick ideal.
\end{proof}

The following proposition was proved in \cite{jsw} for quadratic Wick ideals and in \cite{proams} in general case.
\begin{proposition}\label{wickide}
Let $P\colon\mathcal{H}^{\otimes n}\rightarrow\mathcal{H}^{\otimes
n}$ be a  projection. The subspace
$\mathcal{I}=P(\mathcal{H}^{\otimes n})$ generates a Wick ideal iff
\begin{enumerate}
\item $R_nP=0$ (equality in $\mathcal{H}^{\otimes n}$),
\item $[\mathbf{1}_{\mathcal{H}}\otimes (\mathbf{1}_{\mathcal{H}^{\otimes n}}-P)]T_1T_2\cdots T_n [P\otimes\mathbf{1}_{\mathcal{H}}]=0$ (equality in $\mathcal{H}^{\otimes n+1}$).
\end{enumerate}
Moreover, if  $T$ is braided and $P$ is the  projection onto $\ker R_n$, the second condition holds automatically and hence $\ker R_n$ generates the largest homogenenous Wick ideal of degree $n$.
\end{proposition}
\begin{remark}\label{remincl}
Note, that  the second condition of Proposition \ref{wickide} means
\[
T_1T_2\cdots T_n \bigl(\mathcal{I}\otimes\mathcal{H}\bigr)\subset\mathcal{H}\otimes\mathcal{I}.
\]
\end{remark}
\begin{lemma}\label{ntonplus1}
Let $\mathcal{I}\subset\mathcal{H}^{\otimes n}$ generate a  homogeneous Wick ideal, then
\[
\bigl(\mathbf{1}_{\mathcal{H}^{\otimes( n+1)}}-T_1T_2\cdots T_n\bigr)(\mathcal{I}\otimes\mathcal{H})\subset\ker R_{n+1}.
\]
\end{lemma}
\begin{proof}
Let $X\in\mathcal{I}$. Then $X\in\ker R_n$. Note  that
\[
R_{n+1}=R_n\otimes\mathbf{1}_{\mathcal{H}}+T_1T_2\cdots T_n=\mathbf{1}_{\mathcal{H}^{\otimes (n+1)}}+T_1(\mathbf{1}_{\mathcal{H}}\otimes R_n)
\]

Then
for any $i=1,\ldots,d$ one has
\begin{align*}
R_{n+1}&(\mathbf{1}_{\mathcal{H}^{\otimes (n+1)}}-T_1T_2\cdots T_n)(X\otimes e_i)=\\
&=R_{n+1}(X\otimes e_i)-
R_{n+1}T_1T_2\cdots T_n (X\otimes e_i)=\\
&=(R_n\otimes\mathbf{1}_{\mathcal{H}}+T_1T_2\cdots T_n)(X\otimes e_i)\\
&-
(\mathbf{1}_{\mathcal{H}^{\otimes (n+1)}}+T_1(\mathbf{1}_{\mathcal{H}}\otimes R_n))T_1T_2\cdots T_n (X\otimes e_i)=\\
&=T_1T_2\cdots T_n(X\otimes e_i)-T_1T_2\cdots T_n(X\otimes e_i)\\
& -T_1(\mathbf{1}_{\mathcal{H}}\otimes R_n)T_1T_2\cdots T_n(X\otimes e_i)=0,
\end{align*}
where we used
\[
T_1T_2\cdots T_n(\mathcal{I}\otimes\mathcal{H})\subset\mathcal{H}\otimes{\mathcal{I}}
\subset\mathcal{H}\otimes\ker R_n=\ker(\mathbf{1}_{\mathcal{H}}\otimes R_n).
\]
\end{proof}
\noindent The following corollary is immediate.
\begin{corollary}
If the operator $T$ is braided, then
\[
(\mathbf{1}_{\mathcal{H}^{\otimes (n+1)}}-T_1T_2\cdots T_n)(\ker R_n\otimes\mathcal{H})\subset\ker R_{n+1}.
\]
\end{corollary}

Below we will use the following simple observation
\begin{lemma}
Let  $T$ be braided. Then for any $n\ge 2$ and $k\le n-1$
\[
(T_1T_2\cdots T_n)(T_1T_2\cdots T_k)=(T_2T_3\cdots T_{k+1})(T_1T_2\cdots T_n).
\]
\end{lemma}
\begin{proof}
Evidently it is enough to check that
\[
T_1T_2\cdots T_n T_j=T_{j+1}T_1T_2\cdots T_n,\quad 1\le j\le n-1.
\]
Indeed, since $T_iT_j=T_jT_i$ when $|i-j|\ge 2$ and $T_jT_{j+1}T_j=T_{j+1}T_jT_{j+1}$ we get
\begin{align*}
T_1T_2\cdots T_n T_j & =
T_1T_2\cdots T_{j-1}T_{j}T_{j+1}T_jT_{j+2}\cdots T_n=\\
&=T_1T_2\cdots T_{j-1}T_{j+1}T_jT_{j+1}T_{j+2}\cdots T_n=\\
&=T_{j+1}T_1T_2\cdots T_n.
\end{align*}
\end{proof}

The following proposition  gives a procedure to compute
 generators of certain homogeneous Wick ideals of degree $n+1$ out of generators of Wick ideals of degree $n$ when $T$ is braided.
\begin{proposition}
Let  $T$ be braided and $\mathcal{I}_n\subset\mathcal{H}^{\otimes n}$ generate a homogeneous Wick ideal of degree $n$. Then
\[
\mathcal{I}_{n+1}=(\mathbf{1}_{\mathcal{H}^{\otimes (n+1)}}-T_1T_2\cdots T_n)(\mathcal{I}_n\otimes\mathcal{H})
\]
generates a homogeneous Wick ideal of degree $n+1$.
\end{proposition}
\begin{proof}
According to Lemma \ref{ntonplus1}
\[
(\mathbf{1}_{\mathcal{H}^{\otimes (n+1)}}-T_1T_2\cdots T_n)(\mathcal{I}_n\otimes\mathcal{H})\subset\ker R_{n+1}
\]
so, it remains to prove that
\begin{equation}\label{**}
T_1T_2\cdots T_{n+1}(\mathcal{I}_{n+1}\otimes\mathcal{H})\subset
\mathcal{H}\otimes\mathcal{I}_{n+1}.
\end{equation}
Indeed
\begin{align*}
T_1T_2\cdots T_{n+1}&(\mathcal{I}_{n+1}\otimes\mathcal{H})=T_1T_2\cdots T_{n+1}
(\mathbf{1}_{\mathcal{H}^{\otimes (n+1)}}-T_1T_2\cdots T_n)(\mathcal{I}_n\otimes\mathcal{H}\otimes\mathcal{H})=\\
&=(T_1T_2\cdots T_{n+1}-T_1T_2\cdots T_{n+1}T_1T_2\cdots T_n)
(\mathcal{I}_n\otimes\mathcal{H}\otimes\mathcal{H})=\\
&=(T_1T_2\cdots T_{n+1}-T_2T_3\cdots T_{n+1}T_1T_2\cdots T_{n+1})(\mathcal{I}_n\otimes\mathcal{H}\otimes\mathcal{H})=\\
&=(\mathbf{1}_{\mathcal{H}^{\otimes (n+1)}}-T_2T_3\cdots T_{n+1})T_1T_2\cdots T_{n+1}(\mathcal{I}_n\otimes\mathcal{H}\otimes\mathcal{H})=\\
&=(\mathbf{1}_{\mathcal{H}^{\otimes (n+1)}}-T_2T_3\cdots T_{n+1})T_1T_2\cdots T_n(\mathcal{I}_n\otimes T(\mathcal{H}\otimes\mathcal{H}))\subset\\
&\subset (\mathbf{1}_{\mathcal{H}^{\otimes (n+1)}}-T_2T_3\cdots T_{n+1})T_1T_2\cdots T_n(\mathcal{I}_n\otimes\mathcal{H}\otimes\mathcal{H})\subset\\
&\subset
(\mathbf{1}_{\mathcal{H}^{\otimes (n+1)}}-T_2T_3\cdots T_{n+1})(\mathcal{H}\otimes\mathcal{I}_n\otimes\mathcal{H})=\\
&=\mathcal{H}\otimes (\mathbf{1}_{\mathcal{H}^{\otimes n}}-T_1T_2\cdots T_n)(\mathcal{I}_n\otimes\mathcal{H})=\\
&=\mathcal{H}\otimes\mathcal{I}_{n+1}.
\end{align*}
\end{proof}

Next aim is to describe largest Wick ideals.

\begin{lemma}
Let $T$ satisfy the braid relation. Then
\begin{equation}\label{rntcomm}
 R_{n+1}T_1T_2\cdots T_n=T_1T_2\cdots T_n+T_1^2T_2\cdots T_n
(R_n\otimes\mathbf{1}_{\mathcal{H}})
\end{equation}
\end{lemma}
\begin{proof}
Indeed
\begin{align*}
& R_{n+1}T_1T_2\cdots T_n=\\
&=T_1T_2\cdots T_n+
T_1(\mathbf{1}_{\mathcal{H}^{\otimes n+1}}+T_2+T_2T_3+\cdots +
T_2T_3\cdots T_n)T_1T_2\cdots T_n=\\
&=T_1T_2\cdots T_n+T_1^2T_2\cdots T_n
(\mathbf{1}_{\mathcal{H}^{\otimes n+1}}+T_1+T_1T_2+\cdots+T_1T_2\cdots T_{n-1})
=\\
&=T_1T_2\cdots T_n+T_1^2T_2\cdots T_n (R_n\otimes\mathbf{1}_{\mathcal{H}}).
\end{align*}
\end{proof}
\begin{lemma}\label{rntn}
Let $T$ be braided. Then
\[
R_{n+1}(\mathbf{1}_{\mathcal{H}^{\otimes( n+1)}}-T_1T_2\cdots T_n)=(\mathbf{1}_{\mathcal{H}^{\otimes (n+1)}}-T_1^2T_2\cdots T_n)(R_n\otimes\mathbf{1}_{\mathcal{H}}).
\]
\end{lemma}
\begin{proof}
By the  previous Lemma
\begin{align*}
R_{n+1}&-R_{n+1}T_1T_2\cdots T_n=\\
&=R_n\otimes\mathbf{1}_{\mathcal{H}}+T_1T_2\cdots T_n-T_1T_2\cdots T_n-T_1^2T_2\cdots T_n (R_n\otimes\mathbf{1}_{\mathcal{H}})=\\
&=(1-T_1^2T_2\cdots T_n)(R_n\otimes\mathbf{1}_{\mathcal{H}}).
\end{align*}
\end{proof}


Let  $\mathcal{K}_2=\ker R_2$ and
\[
\mathcal{K}_{m+1}=(\mathbf{1}_{\mathcal{H}^{\otimes (m+1)}}-T_1T_2\cdots T_{m})(\mathcal{K}_m\otimes\mathcal{H}),\quad m\ge 2.
\]
Since by (\ref{**})
\[
\mathcal{K}_{m+1}\subset\mathcal{H}\otimes\mathcal{K}_m+
\mathcal{K}_m\otimes\mathcal{H},
\]\
the Wick ideals generated by $\mathcal{K}_m$, $m\ge 2$,  form a nested sequence
\[
\langle\mathcal{K}_2\rangle\supset\langle\mathcal{K}_3\rangle\supset\cdots\supset\langle \mathcal{K}_m\rangle\supset\cdots
\]
\begin{proposition}\label{keqr}
Suppose that $T$ is braided and for any $m\ge 2$
\[
\ker (\mathbf{1}_{\mathcal{H}^{\otimes (m+1)}}-T_1T_2\cdots T_m)=\{0\}\quad \mbox{and}\quad\ker(\mathbf{1}_{\mathcal{H}^{\otimes (m+1)}}-T_1^2T_2\cdots T_m)=\{0\}.
\]
Then
\[
\mathcal{K}_m=\ker R_m,\ m\ge 2,
\]
and hence  $\mathcal{K}_m$
generates the  largest homogeneous Wick ideals of degree $m$ for any $m\ge 2$.
\end{proposition}
\begin{proof}
Suppose that $\dim\mathcal{H}=d$. If
 $\mathbf{1}_{\mathcal{H}^{\otimes m+1}}-T_1T_2\cdots T_m$, $m\ge 2$ are invertible, by the definition of $\mathcal{K}_m$ we have
\[
\dim\mathcal{K}_m=d\cdot\dim\mathcal{K}_{m-1}=d^{m-2}\cdot\dim\ker R_2.
\]
As ${\mathcal K}_m\subset \ker R_m$ (by Lemma~\ref{ntonplus1})
 it remains to see  that for any $m\ge 2$ one has
\[
\dim\ker R_{m}=d\cdot\dim\ker R_{m-1}=\ldots=d^{m-2}\dim\ker R_2
\]
But this  immediatly follows from the equality
\[
R_{m+1}(\mathbf{1}_{\mathcal{H}^{\otimes m+1}}-T_1T_2\cdots T_m)=(\mathbf{1}_{\mathcal{H}^{\otimes m+1}}-T_1^2T_2\cdots T_m)(R_m\otimes\mathbf{1}_{\mathcal{H}})
\]
and invertibility of the operators
$\mathbf{1}_{\mathcal{H}^{\otimes m+1}}-T_1T_2\cdots T_m$ and
$\mathbf{1}_{\mathcal{H}^{\otimes m+1}}-T_1^2T_2\cdots T_m$.

Hence, $\dim\ker R_m=\dim\ker\mathcal{K}_m$ and
\[
\mathcal{K}_m=\ker R_m,\quad m\ge 2.
\]

\end{proof}

\begin{lemma}\label{1notinsp}
Let $T$ be braided and $||T_1T_2T_1||=q<1$, $||T||=1$. Then $\ker R_m=\mathcal{K}_m$ for any $m\ge 2$.
\end{lemma}
\begin{proof}
By Propostion~\ref{keqr} it is enough to see  that
\[
1\not\in\sigma(T_1T_2\cdots T_n)\quad\mbox{and}\quad 1\not\in\sigma(T_1^2T_2\cdots T_n).
\]

Indeed, since
$T_iT_j=T_jT_i$, $|i-j|\ge 2$, and $||T_i||=1$, $i=1,\ldots,n$,  we get
\[
(T_1T_2T_3\cdots T_n)^2=(T_1T_2T_1)(T_3T_4\cdots T_nT_2T_3\cdots T_n)
\]
implying
\[
||(T_1T_2\cdots T_n)^2||\le q<1
\]
 and hence $1\not\in\sigma(T_1T_2\cdots T_n)$.

Analogously,
\[
(T_1^2T_2\cdots T_n)^2=T_1(T_1T_2T_1)(T_3T_4\cdots T_nT_1T_2\cdots T_n)
\]
and $||(T_1^2T_2\cdots T_n)^2||\le q<1$ giving $1\notin \sigma(T_1^2T_2\ldots T_n)$.
\end{proof}
In what follows we shall often say ideal ${\mathcal K}_m$ meaning
the ideal generated by ${\mathcal K}_m$.

In general, see Section 3 and Section 4,  largest homogeneous Wick ideals do not coincide with the ideals $\mathcal{K}_m$. However a  direct calculations in {\sc Mathematica} shows that for some Wick algebras, including Wick versions of CCR, twisted CCR, twisted CAR and quonic commutation relations, see \cite{jsw}, the following conjecture is true.
\begin{conjecture}\label{strwide}
If $T$ is braided then
\[
\ker R_{n+1}=(\mathbf{1}_{\mathcal{H}^{\otimes( n+1)}}-T_1T_2\cdots T_n)(\ker R_n\otimes\mathcal{H})+\ker R_{n-2}\otimes\ker R_2.
\]
\end{conjecture}
\section{Homogeneous ideals of Wick version of quon commutation relations}
Here we apply  results of the previous section to get a
description of homogeneous ideals the Wick algebra,
$\mathcal{A}_2^q$, associated with quon commutation relations with
two degrees of freedom, see \cite{mar}. Recall that
$\mathcal{A}_2^q$ is a $*$-algebra generated by elements $a_i$,
$a_i^*$, $i=1,2$, satisfying commutation relations of the form
\begin{align*}
a_i^*a_i&=1+qa_ia_i^*,\ i=1,2,\\
a_1^*a_2&=\lambda a_2 a_1^*,
\end{align*}
where $q$, $\lambda$ are parameters such that  $0<q<1$,
$|\lambda|=1$. In this case $\dim\mathcal{H}=2$ and the operator $T$
is given by
\begin{align}\label{tquon}
T e_i\otimes e_i & =q e_i\otimes e_i,\ i=1,2,\nonumber\\
T e_1\otimes e_2 & =\overline{\lambda} e_2\otimes e_1,\quad T e_2\otimes e_1=\lambda e_1\otimes e_2
\end{align}
It is easy to verify that $T$ is braided, $||T||=1$ for any $q\in (0,1)$, $\lambda\in\mathbb{C}$, $|\lambda|=1$ and

\[ \ker
(\mathbf{1}_{\mathcal{H}^{\otimes 2}}+T)=\mathbb C\left<
A=e_2\otimes e_1-\lambda e_1\otimes e_2\right>.
\]

\begin{proposition}\label{quonwide}
Let $T\colon\mathcal{H}^{\otimes 2}\rightarrow\mathcal{H}^{\otimes
2}$ be defined by  (\ref{tquon}) and $\dim\mathcal{H}=2$. Then for
any $m\ge 2$, $\ker R_m=\mathcal{K}_m$  is the largest homogeneous
Wick ideal of degree
$m$.
\end{proposition}
\begin{proof}
By Lemma~\ref{1notinsp} it is enough to show  that
$||T_1T_2T_1||<1$. Indeed, it is  easy to see that for the standard
orthonormal basis of $\mathcal{H}^{\otimes 3}$ one has
\begin{align*}
T_1T_2T_1 e_i\otimes e_i\otimes e_i&=q^3 e_i\otimes e_i\otimes e_i,\quad i=1,2\\
T_1T_2T_1 e_1\otimes e_1\otimes e_2&=q\overline{\lambda}^2 e_2\otimes e_1\otimes e_1,\quad
T_1T_2T_1 e_2\otimes e_1\otimes e_1=q\lambda^2 e_1\otimes e_1\otimes e_2\\
T_1T_2T_1 e_2\otimes e_2\otimes e_1&=q\lambda^2 e_1\otimes e_2\otimes e_2,\quad
T_1T_2T_1 e_1\otimes e_2\otimes e_2=q\overline{\lambda}^2 e_2\otimes e_2\otimes e_1\\
T_1T_2T_1 e_1\otimes e_2\otimes e_1&=q\  e_1\otimes e_2\otimes e_1,\quad\ \
T_1T_2T_1 e_2\otimes e_1\otimes e_2=q\ e_2\otimes e_1\otimes e_2.
\end{align*}
Hence $||T_1T_2T_1||=q<1$.
\end{proof}
\begin{remark}
\begin{enumerate}
\item For Wick quonic relations with
three generators Lemma~\ref{1notinsp} cannot be applied, since in this case $||T_1T_2T_1||=1$ . However, since $T$ is a braided contraction we have by Proposition \ref{cubide}
$$\ker R_3=(\mathbf{1}_{\mathcal{H}^{\otimes 3}}-T_1T_2)(\ker
R_2\otimes\mathcal{H})
$$
and  one can apply
Proposition \ref{keqr} to show that in this case
$\mathcal{K}_m=\ker R_m$, $m\ge 2$ as well.
\item Computations in {\sc Mathematica} show that for  Wick qounic relations with four or more generators
the ideals $\mathcal{K}_m$ do not coincide with $\ker R_m$ for
$m>3$.
\end{enumerate}
\end{remark}
\subsection{$*$-Representations of $\mathcal{A}_2^q$, annihilating homogeneous ideals}
In this section  we show that any $*$-representation of the Wick
quonic relations annihilating $\mathcal{K}_m$ for some fixed $m\ge
2$ annihilates the ideal $\mathcal{K}_2$.

First we recall that for any bounded $*$-representation $\pi$ of
$\mathcal{A}_2^q$ one has $\pi(\mathcal{K}_2)=0$, see \cite{osam}.
Indeed, it easy to verify, that if $A=a_2a_1-\lambda a_1a_2$, then
\[
a_1^*A=\lambda qAa_1^*,\quad a_2^* A=\overline{\lambda}q Aa_2^*
\]
implying that $A^*A=q^2AA^*$. Evidently, the only bounded operator
$A$ satisfying  such relation is the zero one.
\begin{proposition}
Let $\pi$ be an irreducible $*$-representation (possibly unbounded)
of $\mathcal{A}_2^q$ such that $\pi(\mathcal{K}_m)=\{0\}$ for some
$m\ge 3$. Then $\pi(A)=0$ and hence $\pi(\mathcal K_2)=0$.
\end{proposition}
\begin{proof}
By  Propositions \ref{quonwide} for any $m\ge 3$ the
ideal $\mathcal{K}_m$ coincides with the largest homogeneous ideal
of degree $m$.

Let $m=2k$, for some $k>1$. Then, since the  product of homogeneous
Wick ideals is a homogeneous Wick ideal, we get
\[
(\ker R_2)^{\otimes k}\subset\ker R_{2k}=\mathcal{K}_m.
\]
So if $\pi(\mathcal{K}_m)=\{0\}$, then $\pi(A^k)=0$ and hence
$\ker\pi(A)\ne\{0\}$.  Further, $A^*A=q^2 AA^*$ implies that
$\ker\pi(A)=\ker\pi(A^*)$ and from
\[
Aa_1^*=\overline{\lambda}q^{-1}a_1^*A,\ Aa_2^*=\lambda q^{-1}a_2^*A,\ A^*a_1=\overline{\lambda}qa_1A^*,\ A^*a_2=\lambda q a_2A^*
\]
we obtain that $\ker\pi(A)=\ker\pi(A^*)$ is invariant with respect to $\pi(a_i)$ and $\pi(a_i^*)$, $i=1,2$.
 Thus if $\pi$ is irreducible, $\pi(A)=\{0\}$.

Suppose now that $\pi(\mathcal{K}_m)=\{0\}$ and $m=2k+1$, for fixed
$k\ge 1$. Then as above $A^{k+1}\in\mathcal{K}_{m+1}$. Since
$\langle\mathcal{K}_{m+1}\rangle\subset\langle \mathcal{K}_m\rangle$  we get $\pi(A^{k+1})=0$ and
repeating the  arguments from the previous paragraph we obtain
$\pi(A)=0$.
\end{proof}

We refer the reader to \cite{schmudgen} for definitions and facts
about unbounded $*$-representations of $*$-algebras. Note that such
representations can be rather complicated and one usually restricts
oneself to a subclass of "well-behaved" representations. For Lie
algebras natural well-behaved  representations are
integrable representations i.e. those which can be integrated to a
unitary representation of the corresponding Lie group (see for
example \cite[Section 10]{schmudgen}).

\section{$*$-Representations of Wick version of CCR annihilating homogenenous ideals}
In this Section we consider a Wick version of CCR, denoted below by
$\mathcal{A}_d^0$, and  given by
\[
\mathcal{A}_d^0=\mathbb{C}\left<a_i,\ a_i^* \mid
a_i^*a_j=\delta_{ij}\mathbf{1}+a_ja_i^*,\ i,j=1,\ldots, d\right>.
\]
In this case $T$ is the  flip operator
\[
T e_i\otimes e_j=e_j\otimes e_i,\ i,j=1,\ldots,d
\]
and the largest quadratic ideal $\mathcal{K}_2=\ker R_2$ is
generated by the elements
\[
A_{ij}=e_j\otimes e_i-e_i\otimes e_j,\quad i\ne j,\ i,j=1,\ldots,d.
\]
The action of the operator $T_1T_2\cdots T_k$ on a product of the form
$B\otimes e_i$, $B\in{\mathcal{H}^{k}}$, $i=1,\ldots,d$, is the
following
\[
(T_1T_2\cdots T_k)(B\otimes e_i)=e_i\otimes B,\quad i=1,\ldots,d.
\]
Thus if the homogeneous Wick ideal $\mathcal{K}_m$ is generated by a
family $\{B_j,\ j\in\mathcal{J}\}$, then
\[
\mathcal{K}_{m+1}=\bigl<e_i\otimes B_j-B_j\otimes e_i,\ i=1,\ldots,d,\ j\in\mathcal{J}\bigr>
\]

Recall that
\[
e_i^*\otimes B_j=\mu_0(e_i^*)\bigl(R_m B_j+\sum_{k=1}^d T_1T_2\cdots
T_m (B_j\otimes e_k)\otimes e_k^*\bigr),\ i=1,\ldots,d,\
j\in\mathcal{J}.
\]
Since
\[
T_1T_2\cdots T_m(B_j\otimes e_k)=e_k\otimes B_j,\quad R_nB_j=0,
\]
and $\mu_0(e_i^*)e_k\otimes X=\delta_{ik}X$ for any  $X\in\mathcal{T}(\mathcal{H})$,  we get
\[
e_i^*\otimes B_j=B_j\otimes e_i^*,\quad i=1,\ldots,d,\ j\in\mathcal{J}.
\]

In other words if we consider the quotient of $\mathcal{A}_d^0$ by
the homogeneous Wick ideal $\mathcal{K}_{m+1}$ we obtain the
following commutation relations between generators of the algebra
and generators of the ideal $\mathcal{K}_{m}$
\[
a_i^*B_j=B_ja_i^*,\quad a_iB_j=B_ja_i,\ i=1,\ldots,d,\ j\in\mathcal{J}.
\]

We intend to study representations of $\mathcal{A}_2^0$ annihilating
the ideals $\mathcal{K}_m$, $m=2,3,4$.

\subsection{Representations of $\mathcal{A}_2^0$ annihilating quadratic and cubic ideals}
Below we assume $d=2$. The quadratic ideal $\mathcal{K}_2$ is
generated by $a_1\otimes a_2 - a_2\otimes a_1$ and the quotient
$\mathcal{A}_2^0/\mathcal{K}_2$ is the Weyl algebra with two degrees
of freedom. Note that it is a quotient of the universal enveloping
of the Heisenberg algebra. The unique irreducible well-behaved
representation of the Weyl algebra (by well-behaved we mean a
representation which can be integrated to a unitary representation
of the Heisenberg Lie group), 
is the Fock representation: the space of the
representation is
$\mathcal{H}=l_2(\mathbb{Z}_{+})\otimes l_2(\mathbb{Z}_{+})$ and
\[
a_1=a\otimes\mathbf{1},\quad a_2=\mathbf{1}\otimes a,
\]
where $a e_n=\sqrt{n+1}e_{n+1}$, $n\in\mathbb{Z}_{+}$, and $\{e_n,\
n\in\mathbb{Z}_+\}$ is the standard orthonormal basis in $l_2(\mathbb{Z}_{+})$.

Now we study irreducible representations of
$\mathcal{A}_2^0$ which annihilate the ideal $\mathcal{K}_3$.
The ideal $\mathcal{K}_3$ is generated by the elements
\[
Aa_1-a_1A,\quad Aa_2-a_2A
\]
with $A=a_2a_1-a_1a_2$.

Since $a_i^* A=A a_i^*$, $i=1$, $2$, we conclude that $A$ belongs to the center of the quotient $\mathcal{A}_2^0/\mathcal{K}_3$.

For a well-behaved irreducible representation $\pi$, we assume that
$A$ commutes with $a_i$, $a_i^*$ strongly (i.e. $A$ is closable on
the domain of the representation and if $A=U|A|$ is the polar
decompostion of $A$ then $U$ and all spectral projections of $|A|$
belongs to the strong commutant of the family $\{a_1,a_1^*, a_2,
a_2^*\}$, \cite{schmudgen}) 
and by the Schur lemma we
have $A=x\mathbf1$, $x\in \mathbb C$ (we denote the operators of the
representation by the same letters as the corresponding elements of
the algebra). Thus, the problem of classification of such
irreducible representations is reduced to the classification of
irreducible representations  of the following family of commutation
relations
\begin{align}\label{ccrk3}
&a_i^*a_i-a_ia_i^*=\mathbf1,\quad i=1,2,\nonumber\\
&a_1^*a_2=a_2a_1^*\quad
a_2a_1-a_1a_2=x\mathbf1.
\end{align}
Denote by $A_{2,x}$ the $*$-algebra generated by relations (\ref{ccrk3}) and by $A_{2,0}$  the $*$-algebra generated by CCR with two degrees of freedom.
\begin{proposition}\label{axa2}
The $*$-algebras $A_{2,x}$ and $A_{2,0}$ are isomorphic for any
$x\in\mathbb{C}$.
\end{proposition}
\begin{proof}
For any fixed $x\in\mathbb{C}$ let
\[
d_1=a_1\quad\mbox{and}\quad d_2=\Bigl(1+|x|^2\Bigr)^{-\frac{1}{2}}a_2-xa_1^*.
\]
Then it is easy to verify that $d_1$, $d_2$ generate $A_{2,x}$ and
\begin{equation}\label{2ccr}
d_i^*d_i-d_id_i^*=1,\ i=1,2,\quad d_1^*d_2=d_2d_1^*,\quad d_2d_1=d_1d_2.
\end{equation}

Conversely, let $c_1$, $c_2$ be generators of $A_{2,0}$  satisfying
(\ref{2ccr}). Put
\[
b_1=c_1,\quad b_2=\Bigl(1+|x|^2\Bigr)^{\frac{1}{2}}c_2+xc_1^*
\]
Then $b_1$, $b_2$ satisfy (\ref{ccrk3}) and generate $A_{2,0}$. Hence $A_{2,x}\simeq A_{2,0}$.
\end{proof}
It follows from the uniqueness of irreducible well-behaved
representation of CCR with two degrees of freedom that there exists
a unique, up to a unitary equivalence, irreducible representation
of (\ref{2ccr}) defined on $l_2(\mathbb{Z}_{+})^{\otimes 2}$ by the
formulas
\begin{equation*}
d_1=a\otimes\mathbf{1},\quad d_2=\mathbf{1}\otimes a.
\end{equation*}

Below by {\it well-behaved} representation of $A_{2,x}$ we mean a
well-behaved representation of $A_{2,0}\simeq A_{2,x}$. Applying
Proposition \ref{axa2} we get the following result.
\begin{theorem}\label{repcubid}
For any $x\in\mathbb{C}$ there exists a unique, up to  unitary
equivalence, irreducible well-behaved representation of $A_{2,x}$
given  by
\begin{align*}
a_1&=a\otimes\mathbf{1},\\
a_2&=\sqrt{1+|x|^2}\mathbf{1}\otimes a+x a^*\otimes\mathbf{1}.
\end{align*}
\end{theorem}

Evidently in the case $x=0$ we get the Fock representation, annihilating $\mathcal{K}_2$.
\subsection{Representations annihilating $\mathcal{K}_4$}
Let us describe  representations of $\mathcal{A}_2^0$ which
annihilate the ideal $\mathcal{K}_4$. Recall that
\[
\mathcal{K}_4=\langle B_i a_j-a_j B_i,\quad i,j=1,2\rangle,
\]
where $B_i=A a_i-a_iA$, $i=1,2$, are generators of $\mathcal{K}_3$.
  Since
\[
a_j^*B_i=B_ia_j^*,\ i=1,2,
\]
the elements $B_1$, $B_2$ belong to the center of the quotient $\mathcal{A}_2^0/\mathcal{K}_4$. Identifying again the elements with their images in a representation
$\pi$ annihilating  $\mathcal K_4$ we require that for a well-behaved irreducible representation
\[
B_1=Aa_1-a_1A=x_1\mathbf{1},\quad B_2=Aa_2-a_2A=x_2\mathbf{1}
\]
for some $x_1$, $x_2\in\mathbb{C}$.
Note  also that in $\mathcal{A}_2^0$ we have $a_i^*A=Aa_i^*$, $i=1,2$.

\subsubsection{Representations with $x_1\ne 0$.}
Fix  $(x_1,x_2)\in\mathbb{C}^2$ with $x_1\ne 0$ and consider the $*$-algebra $A_{x_1,x_2}$, generated by elements $a_1$, $a_2$, $A$ satisfying the following commutation relations
\begin{gather}
a_i^*a_i-a_ia_i^*=1,\nonumber\\
a_1^*a_2=a_2a_1^*,\quad A=a_2a_1-a_1a_2,
\label{abasic}
\\
Aa_i-a_iA=x_i1,\quad a_i^*A=Aa_i^*,\quad i=1,2 \nonumber.
\end{gather}
Let
\begin{align}\label{c1c2c3}
d_1&=a_1\nonumber\\
d_2&=|x_1|^{-1}(A-x_1a_1^*)\\
d_3&=\Bigl(1+\frac{|x_2|^2}{|x_1|^2}\Bigr)^{-\frac{1}{2}}\Bigl(a_2+
\frac{x_2}{|x_1|}d_2^*
-\frac{\overline{x}_1}{2} d_2^2-|x_1|d_1^*d_2-
\frac{x_1}{2} (d_1^*)^2\Bigr)\nonumber
\end{align}
Below we show that the elements $d_i$, $i=1,2,3$, generate
$A_{x_1,x_2}$ and satsify CCR with three degrees of freedom.

First we establish some commutation relations between $a_i$ and
$d_j$, $i,j=1,2$.
\begin{lemma}
The elements $a_1$, $a_2$, $d_1$, $d_2$ satisfy the following relations
\begin{align}\label{adcomrel}
d_1^*a_2&=a_2d_1^*,\nonumber\\
a_2d_1-d_1a_2&=|x_1|d_2+x_1d_1^*,\nonumber\\
a_2^*d_2&=d_2a_2^*+x_1d_2^*+|x_1|d_1,\\
a_2d_2&=d_2a_2-\frac{x_2}{|x_1|}.\nonumber
\end{align}
\end{lemma}
\begin{proof}
The first two relations follow directly from the definition of
$d_1$, $d_2$ and (\ref{abasic}). Further
\begin{align*}
|x_1|a_2d_2&=a_2A-x_1a_2a_1^*=Aa_2-x_2-x_1a_1^*a_2=\\
&=(Aa_2-x_1a_1^*)a_2-x_2=|x_1|d_2a_2-x_2,
\end{align*}
and
\begin{align*}
|x_1|a_2^*d_2&=a_2^*A-x_1a_2^*a_1^*=Aa_2^*-x_1(a_1^*a_2^*-A^*)=\\
&=(A-x_1a_1^*)a_2^*+x_1A^*=|x_1|d_2a_2^*+x_1(|x_1|d_2^*+\overline{x}_1d_1)=\\
&=|x_1|(d_2a_2^*+x_1d_2^*+|x_1|d_1).
\end{align*}
\end{proof}
\begin{lemma}\label{ax1x2toa3}
The elements $d_i$, $d_i^*$, $i=1,2,3$, generate $A_{x_1,x_2}$ and
satisfy CCR with three degrees of freedom, i.e. for any $i=1,2,3$
and $i\ne j$
\begin{equation}\label{cbasic}
d_i^*d_i-d_id_i^*=1,\quad d_i^*d_j=d_jd_i^*,\quad d_id_j=d_jd_i.
\end{equation}
\end{lemma}
\begin{proof}
It  easily follows from  (\ref{c1c2c3}) that
\begin{align}\label{a1a2a3}
a_1&=d_1, \nonumber\\
A&=|x_1|d_2+x_1d_1^*,\\
a_2&=\Bigl(1+\frac{|x_2|^2}{|x_1|^2}\Bigr)^{\frac{1}{2}}d_3-
\frac{x_2}{|x_1|}d_2^*+
\frac{\overline{x}_1}{2}d_2^2+|x_1| d_1^*d_2+\frac{x_1}{2}(d_1^*)^2\nonumber
\end{align}
proving that $A_{x_1,x_2}$ is generated by $d_1$, $d_2$, $d_3$.

Further
\begin{align*}
|x_1|d_2d_1&=(A-x_1a_1^*)a_1=Aa_1-x_1(1+a_1a_1^*)=\\
&=a_1A+x_1-x_1-x_1a_1a_1^*=a_1(A-x_1a_1^*)=|x_1|d_1d_2
\end{align*}
\[
|x_1|d_1^*d_2=a_1^*(A-x_1a_1^*)=Aa_1^*-x_1(a_1^*)^2=(A-x_1a_1^*)a_1^*=
|x_1|d_2d_1^*
\]
Now let us check that $d_2^*d_2-d_2d_2^*=1$
\begin{align*}
|x_1|^2d_2^*d_2&=(A^*-\overline{x}_1a_1)(A-x_1a_1^*)=\\
&=A^*A-x_1A^*a_1^*-\overline{x}_1a_1A+|x_1|^2a_1a_1^*=\\
&=AA^*-x_1(a_1^*A^*-\overline{x}_1)-\overline{x}_1(Aa_1-x_1)
+|x_1|^2(a_1^*a_1-1)=\\
&=AA^*-x_1a_1^*A^*-\overline{x}_1Aa_1+|x_1|^2a_1^*a_1+|x_1|^2=\\
&=(A-x_1a_1^*)(A^*-\overline{x}_1a_1)+|x|_1^2=|x_1|^2(1+d_2d_2^*).
\end{align*}
Here we use the evident fact that $AA^*=A^*A$.

The relation $d_1^*d_3=d_3d_1^*$ follows immediately from the definition of $d_3$ and the commutation relations between $d_1^*$ and $d_2$, $d_2^*$. Using this commutation again as well as relations (\ref{adcomrel}) we get
\begin{align*}
\sqrt{1+\frac{|x_2|^2}{|x_1|^2}}(d_1d_3-d_3d_1)=&d_1a_2-a_2d_1+
|x_1|(d_1^*d_1-d_1d_1^*)d_2+\\
&+\frac{x_1}{2}((d_1^*)^2d_1-d_1(d_1^*)^2)=\\
=&-|x_1|d_2-x_1d_1^*+|x_1|d_2+\frac{x_1}{2}2d_1^*=0,
\end{align*}
\begin{align*}
\sqrt{1+\frac{|x_2|^2}{|x_1|^2}}(d_2^*d_3-d_3d_2^*)=&d_2^*a_2-a_2d_2^*-\\
&-\frac{\overline{x}_1}{2}(d_2^*d_2^2-d_2^2d_2^*)
-|x_1|(d_2^*d_2-d_2d_2^*)d_1^*=\\
&=\overline{x}_1d_2+|x_1|d_1^*-\frac{\overline{x}_1}{2}2d_2-|x_1|d_1^*=0
\end{align*}
and
\begin{align*}
\sqrt{1+\frac{|x_2|^2}{|x_1|^2}}(d_2d_3-d_3d_2)&=d_2a_2-a_2d_2+
\frac{x_2}{|x_1|}(d_2d_2^*-d_2^*d_2)=\\
&=\frac{x_2}{|x_1|}-\frac{x_2}{|x_1|}=0.
\end{align*}

Finally, since $d_3d_i=d_id_3$, $d_i^*d_3=d_3d_i^*$, $i=1,2$ one has
\begin{align*}
1=a_2^*a_2-a_2a_2^*&=\bigl(1+\frac{|x_2|^2}{|x_1|^2}\bigr)(d_3^*d_3-d_3d_3^*)-
\frac{|x_2|^2}{|x_1|^2}(d_2^*d_2-d_2d_2^*)+\\
&+\frac{|x_1|^2}{4}((d_2^*)^2d_2^2-d_2^2(d_2^*)^2)+
\frac{|x_1|^2}{4}(d_1^2(d_1^*)^2-(d_1^*)^2d_1^2)+\\
&+|x_1|^2(d_2^*d_2d_1d_1^*-d_1^*d_1d_2d_2^*)+\\
&+\frac{\overline{x}_1|x_1|}{2}d_1(d_2^*d_2^2-d_2^2d_2^*)+
\frac{\overline{x}_1|x_1|}{2}d_2(d_1^2d_1^*-d_1^*d_1^2)+\\
&+\frac{x_1|x_1|}{2}d_2^*(d_1(d_1^*)^2-(d_1^*)^2d_1)+
\frac{x_1|x_1|}{2}d_1^*((d_2^*)^2d_2-d_2(d_2^*)^2)=\\
&=\bigl(1+\frac{|x_2|^2}{|x_1|^2}\bigr)(d_3^*d_3-d_3d_3^*)-
\frac{|x_2|^2}{|x_1|^2}+\\
&+\frac{|x_1|^2}{4}(2+4d_2d_2^*)-
\frac{|x_1|^2}{4}(2+4d_1d_1^*)+|x_1|^2(d_1d_1^*-d_2d_2^*)+\\
&+\frac{\overline{x_1}|x_1|}{2}2d_1d_2-\frac{\overline{x_1}|x_1|}{2}2d_2d_1
-\frac{x_1|x_1|}{2}2d_2^*d_1^*+\frac{x_1|x_1|}{2}2d_1^*d_2^*=\\
&=\bigl(1+\frac{|x_2|^2}{|x_1|^2}\bigr)(d_3^*d_3-d_3d_3^*)-
\frac{|x_2|^2}{|x_1|^2}
\end{align*}
showing that $d_3^*d_3-d_3d_3^*=1$.
\end{proof}

Denote by $A_3$ the $*$-algebra generated by CCR with $3$ degrees of
freedom and denote by $c_1$, $c_2$, $c_3$ the canonical generators
of $A_3$. Construct  elements $b_1$, $b_2$, $B$ of $A_3$ using
formulas (\ref{a1a2a3}).
\begin{lemma}\label{a3toax1x2}
The elements $b_1$, $b_2$, $B$ satisfy (\ref{abasic}) and generate
$A_3$.
\end{lemma}
\begin{proof}
It is evident that one can express $c_i$, $i=1,2,3$ via $b_1$,
$b_2$, $B$ using (\ref{c1c2c3}) with $b_1$, $b_2$, $B$ instead of
$a_1$, $a_2$, $A$. So $A_3$ is generated by $b_1$, $b_2$, $b_3$.

Let us show that $b_1$, $b_2$, $B$ satisfy (\ref{abasic}).

Indeed, it is a moment of reflection to see that $b_1^*b_2=b_2b_1^*$
and $b_1^*b_1-b_1b_1^*=1$. Further
\begin{align*}
b_2b_1-b_1b_2&=|x_1|c_1^*c_2c_1+\frac{x_1}{2}(c_1^*)^2c_1-|x_1|c_1c_1^*c_2-
\frac{x_1}{2}c_1(c_1^*)^2=\\
&=|x_1|(c_1^*c_1-c_1c_1^*)c_2+\frac{x_1}{2}((c_1*)^2c_1-c_1(c_1^*)^2)=\\
&=|x_1|c_2+\frac{x_1}{2}2c_1^*=|x_1|c_2+x_1c_1^*=B,
\end{align*}
\begin{align*}
Bb_1&=|x_1|c_2c_1+x_1c_1^*c_1=|x_1|c_2c_1+x_1(1+c_1c_1^*)=\\
&=|x_1|c_1c_2+x_1c_1c_1^*+x_1=c_1(|x_1|c_2+x_1c_1^*)+x_1=b_1B+x_1,
\end{align*}
and
\[
Bb_2-b_2B=-\frac{x_2}{|x_1|}|x_1|c_2c_2^*+\frac{x_2}{|x_1|}x_1c_2^*c_2=
x_2(c_2^*c_2-c_2c_2^*)=x_2.
\]

Thus it remains to check that $b_2^*b_2-b_2b_2^*=1$. But  in fact this
was done in Lemma \ref{ax1x2toa3}, when we checked  that the
relation $d_3^*d_3-d_3d_3^*=1$ is satisfied.
\end{proof}

Using Lemma \ref{ax1x2toa3} and Lemma \ref{a3toax1x2} it is easy to
see that the $*$-algebras $A_{x_1,x_2}$ and $A_3$ are isomorphic.
\begin{proposition}
The $*$-algebra $A_{x_1,x_2}$ is isomorphic to the $*$-algebra $A_3$.
\end{proposition}
\begin{proof}

Let  $\phi\colon A_{x_1,x_2}\rightarrow A_3$ be a homomorphism
defined by
\[
\phi(a_i)=b_i,\ i=1,2,\quad \phi(A)=B,
\]
where $b_1$, $b_2$ and $B$ are the generators constructed in Lemma
\ref{a3toax1x2}. Similarily define $\psi\colon A_3\rightarrow
A_{x_1,x_2}$ by
\[
\psi(c_i)=d_i,\quad i=1,2,3,
\]
where $d_i$ are taken from Lemma \ref{ax1x2toa3}.

Then  $\psi\circ\phi=\mathrm{id}_{A_{x_1,x_2}}$ and
$\phi\circ\psi=\mathrm{id}_{A_3}$.
\end{proof}

Therefore in order to study irreducible representations of $A_{x_1,x_2}$
we can work with the generators $d_1,d_2,d_3$. As for the case of
representations annihilating $\mathcal{K}_3$, we say that
a representation of $A_{x_1,x_2}$ with $x_1\ne 0$ is {\it
well-behaved} if the corresponding representation of $A_3\simeq
A_{x_1,x_2}$ is well-behaved. Then from the uniqueness of irreducible
well-behaved $*$-representation of CCR with finite degrees of freedom
we get that the space of representation is
$\mathcal{H}=l_2(\mathbb{Z}_{+})^{\otimes 3}$ and
\[
d_1=a\otimes \mathbf1\otimes \mathbf1,\ d_2=\mathbf1\otimes a\otimes \mathbf1,\ d_3=\mathbf1\otimes \mathbf1\otimes a.
\]
Returinig to the generators $a_1$, $a_2$, $a_3$ using  (\ref{c1c2c3}) we get the
following result.
\begin{theorem}
For any $(x_1,x_2)\in\mathbb{C}^2$ with $x_1\ne 0$ there exists a
unique, up to a unitary equivalence, well-behaved irreducible
representation of $A_{x_1,x_2}$ defined on the generators by the
following formulas
\begin{align*}
a_1=&a\otimes \mathbf1\otimes \mathbf1,\\
a_2=&\sqrt{1+\frac{|x_2|^2}{|x_1|^2}}\ \mathbf1\otimes \mathbf1\otimes a-
\frac{x_2}{|x_1|}\mathbf1\otimes a^*\otimes \mathbf1+\frac{\overline{x}_1}{2}\mathbf1\otimes a^2\otimes \mathbf1+\\
&\hskip 3,9cm+|x_1|a^*\otimes a\otimes \mathbf1+\frac{x_1}{2}(a^*)^2\otimes \mathbf1\otimes \mathbf1,\\
A=&|x_1|\mathbf1\otimes a\otimes \mathbf1+x_1a^*\otimes \mathbf1\otimes \mathbf1.
\end{align*}
\end{theorem}

\subsubsection{Representations with $x_1=0$}
Let  $x_1=0$ and $x_2\ne 0$. As in the previous case we have $A_{0,x_2}\simeq A_3$. To see this we express the generators $a_1$, $a_2$, $a_3$ via the generators $d_1$, $d_2$ and $d_3$ of CCR using formulas (\ref{a1a2a3}) with $a_2$, $-a_1$  instead of $a_1$,  $a_2$ respectively, exchanging  $x_1$ with $x_2$ and letting then $x_1=0$. For this we observe that  $(-a_1)a_2-a_2(-a_1)=A$, and
$Aa_2-a_2A=x_2$.
Hence we get the following result.
\begin{theorem}
For any $x_2\in\mathbb{C}$, $x_2\ne 0$, there exists a unique, up to
a unitary equivalence, irreducible well-behaved $*$-representation
of $A_{0,x_2}$, defined by the following formulas
\begin{align*}
a_2&=a\otimes \mathbf1\otimes \mathbf1,\\
a_1&=-\bigl(\mathbf1\otimes \mathbf1\otimes a
+\frac{\overline{x}_2}{2}\mathbf1\otimes a^2\otimes \mathbf1+\\
&\hskip 3,9cm+|x_2|a^*\otimes a\otimes \mathbf1+\frac{x_2}{2}(a^*)^2\otimes \mathbf1\otimes \mathbf1\bigr)\\
A&=|x_2|\mathbf1\otimes a\otimes \mathbf1+x_2a^*\otimes \mathbf1\otimes \mathbf1
\end{align*}
\end{theorem}

If both $x_1=0$ and $x_2=0$, then $Aa_i=a_iA$, $i=1,2$ and hence  the cubic ideal $\mathcal{K}_3$ is annihilated. In this case irreducible well-behaved represenations are described in Theorem \ref{repcubid}.

\subsection{Concluding remarks}
Note that  our result shows in particular that in the case of $\mathcal{A}_2^0$ the ideal $\mathcal{K}_4$ does
not coincide with $\mathcal{I}_4$, the largest homogeneous ideal of degree $4$. Indeed, as noted above
$\mathcal{K}_2\otimes\mathcal{K}_2\subset\mathcal{I}_4$. So, if a representation $\pi$ annihilates
$\mathcal{I}_4$, then $\pi(A^2)=0$. Since $A$ is a normal element we immediately have $\pi(A)=0$.
However the representation that we constructed above has the property that $\pi(\mathcal{K}_4)=\{0\}$
but $\pi(A)\ne 0$. Thus $\mathcal{K}_4\ne\mathcal{I}_4$.

\section*{Acknowledgements} The work on this paper was supported by DFG grant SCHM1009/4-1. The paper was initiated during the visit of V. Ostrovskyi, D. Proskurin and L.~Turowska to Leipzig University, the warm hospitality and stimulating atmosphere are gratefully acknowledged. We are also indebdet deeply to D.~Neiter for performing computations in {\sc Mathematica}.

\end{document}